\documentclass[a4paper,11pt]{article}
\usepackage[a4paper,  total={6in, 10in}]{geometry}
\usepackage[utf8]{inputenc}
\usepackage{amssymb,amsthm,mathtools}
\usepackage{enumitem}
\usepackage{hyperref}
\usepackage[charter]{mathdesign}
\usepackage{color}

\DeclareFontFamily{U}{BOONDOX-calo}{\skewchar\font=45 }
\DeclareFontShape{U}{BOONDOX-calo}{m}{n}{
  <-> s*[1.05] BOONDOX-r-calo}{}
\DeclareFontShape{U}{BOONDOX-calo}{b}{n}{
  <-> s*[1.05] BOONDOX-b-calo}{}
\DeclareMathAlphabet{\mathcalboondox}{U}{BOONDOX-calo}{m}{n}
\SetMathAlphabet{\mathcalboondox}{bold}{U}{BOONDOX-calo}{b}{n}
\DeclareMathAlphabet{\mathbcalboondox}{U}{BOONDOX-calo}{b}{n}
\newcommand{\mcb}[1]{{\mathcalboondox #1}}

\newtheorem{theorem}{Theorem}[section]

\newtheorem{lemma}[theorem]{Lemma}
\newtheorem{proposition}{Proposition}[section]
\newtheorem{definition}{Definition}[section]
\newtheorem{remark}{Remark}[section]

\newcommand{\acks}{\textbf{Acknowledgements.}}

\usepackage{tikz}
\usetikzlibrary{shapes.misc}
\usetikzlibrary{arrows.meta}
\tikzset{cross/.style={cross out, draw=black, minimum size=2*(#1-\pgflinewidth), inner sep=0pt, outer sep=0pt},
cross/.default={0.2cm}}
\definecolor{blue-violet}{rgb}{0.54, 0.17, 0.89}
\usetikzlibrary{decorations.pathreplacing}

\newcommand{\pat}[1]{{\textcolor{black}{{ #1}}}}

\title{Hydrodynamical behavior for the generalized symmetric exclusion with open boundary}
\author{C. Franceschini \thanks{Instituto Superior T\'ecnico, Department of Mathematics, Av. Rovisco Pais 1, 1049-001, Lisbon. E-mail: {\tt chiara.fraceschini@tecnico.ulisboa.pt}}
\and P. Gon\c calves\thanks{Instituto Superior T\'ecnico, Department of Mathematics, Av. Rovisco Pais 1, 1049-001, Lisbon. E-mail: {\tt pgoncalves@tecnico.ulisboa.pt}}\and B. Salvador\thanks{Instituto Superior T\'ecnico, Department of Mathematics, Av. Rovisco Pais 1, 1049-001, Lisbon. E-mail: {\tt beatriz.salvador@tecnico.ulisboa.pt}}}

\begin{document}

\maketitle

\begin{abstract}{We analyze the generalized symmetric exclusion process, which allows at most $\alpha$ particles per site, and we put it in contact with stochastic reservoirs whose strength is regulated by a parameter $\theta\in\mathbb R$. We prove that the hydrodynamic behavior is given by the heat equation and depending on the value of $\theta$, the equation is suplemented with different boundary conditions. Setting $\alpha = 1$ we find the results known in  \cite{baldasso2017exclusion,byronMPRF} for the symmetric simple exclusion process.}
\end{abstract}



\newcommand{\ABS}[1]{\left(#1\right)} 
\newcommand{\veps}{\varepsilon} 


\section{Introduction}
A common problem in the field of statistical mechanics is to deduce the evolution laws of a  thermodynamic quantity of interest in a physical system, by analysing the motion of its molecules. Assuming that each molecule behaves as a continuous-time  random walk, it arises a system of stochastic  interacting particles \cite{spitzer}. Among the many studied models is the exclusion process, which has been serving as a toy model for the analysis of mass transport. Its dynamics can be described as follows: each site of a discrete space can have at most $\alpha\in\mathbb N$ particles. After an exponential clock of rate one, one of the particles at a position $x$ jumps to a position $y$ of the lattice according to a transition rate $p(y-x)$. One is interested in analysing the space-time evolution of the density of particles, which is the quantity conserved by the dynamics. When the parameter $\alpha=1$ and $p(1)=p(-1)=1$, the system is known as the symmetric simple exclusion process (SSEP) and, when $\alpha\neq 1$, as the generalized SSEP. 

The model was first  introduced in \cite{gunter}, where it was called partial exclusion. The name generalized exclusion was first introduced in Section 2.4 of \cite{kipnis1998scaling}, for a model with a mixture of zero-range, since several particles can occupy the same site, and simple exclusion, because the rates $r$ only depend on the number of particles at the departure and the arrival sites, and  restricts the number of particles per site. There the choice for $r$ is such that the model is non-gradient, but  here the choice for $r$   is such that the model is gradient. We mention that the macroscopic behavior of the non-gradient version is studied \pat{in} \cite{claudio} for an open boundary which does not scale with the size of the system. 

Our focus in this article is to deduce the space-time evolution of the density of particles when the system is evolving on the discrete set of sites $\{1,\cdots, N-1\}$, that we call bulk. Moreover, we put the process in contact with two boundary reservoirs (modelled by the sites $x=0$ a nd $x=N$) interacting with the bulk at different intensities. More precisely, the boundary rates of injection and absorption of particles  scale with the size of the system, i.e. $N$,  through a parameter $\theta \in \mathbb{R}$. When $\theta<0$ the interaction is fast and when $\theta\geq0$ it is slow. We observe that when $\alpha \neq 1$, this model differs from the usual SSEP because it allows more than one particle per site and this difference, at the level of the microscopic dynamics, results in having a model that is not solvable by a matrix ansatz formulation  and, as a consequence, there is not much information about its non-equilibrium stationary state (NESS). For $\alpha=1$, the matrix ansatz method developed by \cite{derrida} allows getting information on its NESS, which in turn enables one to obtain explicitly the stationary correlations of the system for any value of $\theta\in\mathbb R$, see, for example,  Section 2.2 of \cite{GJMN} and references therein.

The hydrodynamic limit in the case of the exclusion process with open boundary and $\alpha=1$ was analyzed in \cite{baldasso2017exclusion} the slow \pat{case,} and in \cite{byronMPRF} the fast case. Here we \pat{answer a natural question that usually arises which has to do with the extension of the hydrodynamic limit,  when more than one particle is allowed at each site. Towards this, we  extend the hydrodynamic limit  for $\alpha\neq 1$ in both the slow and  fast regimes and we obtain the heat equation with several boundary conditions.}  Our proof follows the same approach as in \cite{baldasso2017exclusion, byronMPRF} i.e.~ it relies on the entropy method introduced \pat{by Guo, Papanicolau, Varadhan} in \cite{GPV}. 
Our result shows that that the entropy method is robust enough in order to prove the  hydrodynamic limit  for the whole family of generalized exclusion models, regardless of the label $\alpha$ and so also for non-integrable models with a fast/slow boundary.
Now we comment a bit on the strategy of the proof. The idea of the argument is to prove tightness of the sequence of  empirical measures associated to the density and then characterize uniquely the limit point. This last characterization is done in two steps. First, we show that the  limiting points are supported on measures absolutely continuous with respect to the Lebesgue measure. Second, we show that the density is the unique weak solution of the  heat equation with diffusive coefficient $\alpha$ given by
$$
\partial_{t}\rho_{t}(u)=\alpha\Delta\, {\rho} _{t}(u) \;,
$$
 where $\rho(t,u)$ denotes the density of particles at time $t\geq 0$ and position $u\in[0,1]$. Depending on the range of the parameter $\theta$, we obtain \pat{either Dirichlet boundary conditions (for $\theta<1$) that fix the value of the profile at the boundary of $[0,1]$, Robin conditions (for $\theta=1$)  that fix the value of the flux at the boundary of $[0,1]$ as beging proportional to the difference of the densities close to boundary, or Neumann boundary conditions (for $\theta>1$), that fixes the aforementioned flux as being null.} This is similar to the hydrodynamic behavior  obtained in \cite{baldasso2017exclusion, byronMPRF} for the case $\alpha=1$. From our results we also obtain theirs.
 
  The main difference of our proof with respect to previous  proofs is that the two replacement lemmas that are necessary in order to obtain the boundary conditions from the exclusion dynamics, are more complicated since the process allows $\alpha$ particles per site. In the case $\alpha=1$,  since $\eta(x)\in\{0,1\}$, the Glauber dynamics flips the value of the configuration at the boundary so that this transformation turns $\eta(x)$ into $1-\eta(x)$, and the bulk dynamics  coincides with the exchange dynamics, which replaces $\eta(x)$ with $\eta(x+1)$ and vice-versa.
  Nevertheless, when there is more than \pat{one} particle per site, the Glauber dynamics injects or removes particles at the boundary and this does no longer  coincides with the flip dynamics; while the bulk dynamics removes one particle from a site $x$ and takes it to $x+1$, which means that $\eta(x)$ is converted into $\eta(x)-1$ and $\eta(x+1)$ to $\eta(x+1)+1$.  This brings additional difficulties when analysing  the  boundary terms.  The way to overcome this difficulty,  is to make a proper splitting of the \pat{state-space} of the process, and write carefully the action of the Glauber dynamics in a way that creation and annihilation  terms can be combined. 
There are two other related problems that we comment on. The first one is the hydrostatic limit, that we can obtain as a consequence of the hydrodynamic limit in the case $\theta=0$ by a simple analysis of the stationary correlation function and also \pat{for $\theta>1$,} by applying the argument developed in \cite{tsunoda}, which consists in analysing the mass of the system in the time scale $N^{1+\theta}$, and this will appear in a future work \cite{FGJS}. The hydrostatics in the other cases is left open.  
 The second problem  is the analysis of the non-equilibrium fluctuations, which is well known for the case $\alpha = 1$, see \cite{GJMN} and references therein. \pat{For $\alpha\neq 1$, that result is completely open. It would be interesting to check out whether  the fact that for $~\alpha\neq 1$ it does not hold $\eta(x)^2=\eta(x)$ (which  destroys many combinatorial properties that hold for the SSEP)   brings any additional terms to the equation governing the  fluctuations  of the system.  }
 We also highlight that our replacement lemmas are the building blocks in order to analyse the fluctuations and the large deviations principle for our model. All these results are work in progress. 

During the writing of this article, we learned that \cite{schiavo2021scaling} obtained both the hydrodynamic and the hydrostatic limit for our model and for the inclusion process (see also \cite{franceschini2020symmetric}) evolving on Lipschitz domains in arbitrary dimensions. Those results were obtained by an approach based on duality but also on mild solutions of the heat equation and convergence of  random walks to \pat{Brownian motions} with different boundary conditions. In any case, we decided to write our short proof of the hydrodynamic limit because our approach does not rely on those arguments. Nevertheless, many details in our proof are hidden since we try to highlight the main differences with respect to the proof for $\alpha=1$.  We also believe that our method of proof of the replacement lemmas is new and could fit other models of interest. About the dimension, we present the results in $d=1$ but we believe that they can be extended to higher  dimensions as, for example, to the set $\{1,\cdots, N-1\}\times \mathbb T_N^{d-1}$, where $\mathbb T_N$ represents the discrete torus.  

Here follows an outline of the article. In Section \ref{section2} we introduce the model and our main result, Theorem \ref{th:hyd_ssep}. In Section \ref{section3} we prove tightness for any range of the parameter $\theta$ and we characterize the limit points as weak solutions of the corresponding partial differential equations. In Section \ref{section4} we establish the replacement lemmas.
\section{The model and statement of results} \label{section2}
\subsection{SEP($\alpha$)}
Fix a parameter $\alpha\in\mathbb N$ and let $\Lambda_N:=\{1,\cdots, N-1\}$.
The exclusion process that we consider in this paper has state space given by  $\Omega_N:=\{0,\cdots, \alpha\}^{\Lambda_N}$ and the elements of $\Omega_N$ are configurations denoted by $\eta$. The number of particles at $x$ is denoted by $\eta(x)$. In this process particles wait an exponential random time after which, \pat{one of them  jumps to a nearest-neighbor site} if, and only if, the destination site has at most $\alpha-1$ particles, otherwise the particle waits a new random time. 
The infinitesimal Markov generator of our process is  denoted by $\mcb{L}_N$ and it is given on $f :\Omega_N\rightarrow{\mathbb R}$ by
\begin{equation}\label{L_SEP(alpha)_def}
\mcb{L}_N f (\eta) = \mcb{L}_{\ell} f (\eta) + \mcb{L}_{bulk} f (\eta) + \mcb{L}_{r} f (\eta),
\end{equation}
where, for all $\eta \in \Omega_N$,
\begin{flalign} \label{generator_left_SEP}
\mcb{L}_{\ell} f (\eta) &= \frac{1}{N^\theta}c_{1,0}(\eta)\Big\{f(\eta^{1,0}) - f(\eta)\Big\} + \frac{1}{N^\theta}c_{0,1}(\eta) \Big\{f(\eta^{0,1}) - f(\eta)\Big\},& \\ \label{generator_bulk_SEP}
\mcb{L}_{bulk} f (\eta) &= \sum_{x = 1}^{N-2} \left[ c_{x,x+1}(\eta) \Big\{ f(\eta^{x,x+1}) - f(\eta) \Big\} + c_{x+1,x}(\eta) \Big\{ f(\eta^{x+1,x}) - f(\eta) \Big\} \right], & 
\\ \label{generator_right_SEP}
\hspace{0.4cm} \mcb{L}_{r} f (\eta) &= \frac{1}{N^\theta} c_{N-1,N}(\eta) \Big\{f(\eta^{N-1,N}) - f(\eta)\Big\} + \frac{1}{N^\theta}c_{N,N-1}(\eta) \Big\{f(\eta^{N,N-1}) - f(\eta)\Big\}.
\end{flalign}
 Above $$
\eta^{x,y}(z) = (\eta(x)-1)\mathbb{1}_{z=x}+(\eta(y)+1)\mathbb{1}_{z=y}+ \eta(z)\mathbb{1}_{z\neq x,y},$$ where $x,y\in \Lambda_N$. For $x,y\in\Lambda_N$  with $y\neq x$ and $|x-y|\leq 1$, the rates are chosen as  $$c_{x,y}(\eta) := \eta(x) [\alpha-\eta(y)].$$

Now we define $$\eta^{0,1}(z)=(\eta(1)+1)\mathbb{1}_{z=1}\mathbb{1}_{\eta(1)\leq \alpha-1}
+\eta(1)\mathbb{1}_{z=1}\mathbb{1}_{\eta(1)=\alpha}+\eta(z)\mathbb{1}_{z\neq 1}.$$
Analogously we define $\eta^{N,N-1}$ by replacing $0$ by $N$ and $1$ by $N-1
$, respectively. 
We also define 
$$\eta^{1,0}(z)=(\eta(1)-1)\mathbb{1}_{z=1}\mathbb{1}_{\eta(1)\geq 1}
+\eta(1)\mathbb{1}_{z=1}\mathbb{1}_{\eta(1)=0}+\eta(z)\mathbb{1}_{z\neq 1}.$$
Analogously we define $\eta^{N-1,N}$ by replacing $1$ by $N-1$ and $0$ by $N
$, respectively. 

Moreover,   for $\epsilon, \gamma, \beta, \delta\in\mathbb R^+$ and $\theta\in\mathbb R$ we define  $$c_{0,1}(\eta)=\epsilon [\alpha-\eta(1)],\quad \quad  
c_{1,0}(\eta)= \gamma\eta(1),$$ $$c_{N-1,N}(\eta)=\beta\eta(N-1) \quad \textrm{and}\quad c_{N,N-1}(\eta)=\delta [\alpha-\eta(N-1)].$$ 

The dynamics of  SEP($\alpha$) is described in the figure below.

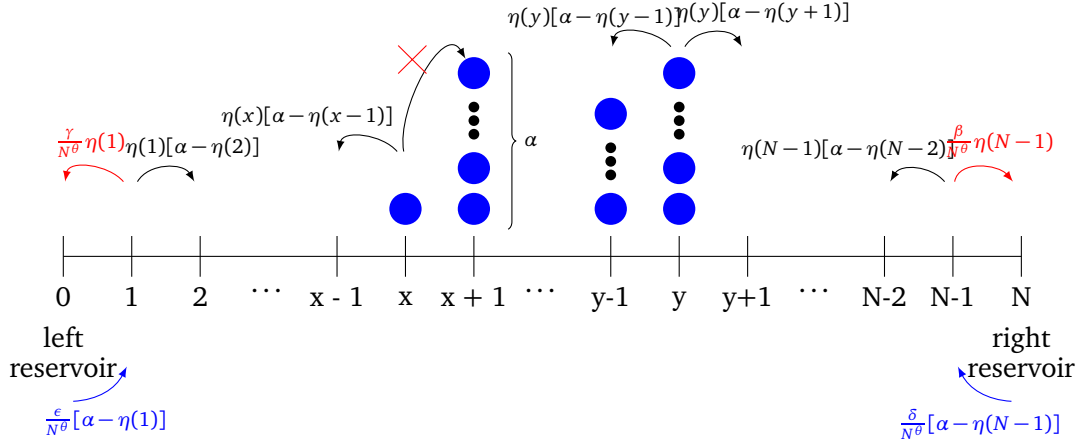
\begin{figure}[htb!]
\centering
\begin{tikzpicture}[scale=0.90]
\node at (5,0.7) (a) [fill,circle,inner sep=0.15cm,color=blue] { };
\node at (6,0.7) (b) [fill,circle,inner sep=0.15cm,color=blue] { };
\node at (6,1.3) (b') [fill,circle,inner sep=0.15cm,color=blue] { };
\node at (6,1.8) (b*) [fill,circle,inner sep=0.05cm,color=black] { };
\node at (6,2) (b'') [fill,circle,inner sep=0.05cm,color=black] { };
\node at (6,2.2) (b**) [fill,circle,inner sep=0.05cm,color=black] { };
\node at (6,2.7) (o) [fill,circle,inner sep=0.15cm,color=blue] { };
\draw (0,0) -- (14,0);
\foreach \x in {0,...,2} \draw (\x, 0.3) -- (\x, -0.3) node (\x) [above=1cm] { } node[below] {\x}; 
\node at (0,0) [below = 0.8cm] {left};
\node at (0,0) [below = 1.2cm] {reservoir};
\node at (0,0) (A) [below = 1.8cm] { };
\node at (1,0) (A') [below = 1.2cm] { };
\node at (3,-0.5) {\ldots};
\draw[-latex] (1) to[out=60,in=120] node[midway,font=\scriptsize,above] {$\hspace{0.7cm}\eta(1)[\alpha-\eta(2)]$} (2);
\draw[-latex,color = red] (1) to[out=130,in=80] node[midway,font=\scriptsize,above,color=red] {$\frac{\gamma}{N^\theta} \eta(1)$} (0);
\draw[-latex, color = blue] (A) to [bend right] node[midway,font=\scriptsize,below=0.1cm,color=blue] {$\frac{\epsilon}{N^\theta} [\alpha-\eta(1)]$} (A');
\foreach \x\y in {4/x - 1,6/x + 1} \draw (\x, 0.3) -- (\x, -0.3) node (\x) [above=1.4cm] { } node [below] {\y};
\foreach \x\y in {5/x} \draw (\x, 0.3) -- (\x, -0.3) node (\x) [above=1.4cm] { } node [below=0.065cm] {\y};
\node at (7,-0.5) {\ldots};
\draw[-latex] (5) to[out=100,in=120] node[midway,font=\scriptsize,above] { } (o);
\draw (5.1,2.9) node[cross,red] { };
\draw[-latex] (5) to[out=130,in=80] node[midway,font=\scriptsize,above] {$\hspace{-1.5cm}\eta(x)[\alpha-\eta(x-1)]$} (4);
\foreach \z\w in {8/y-1,10/y+1} \draw (\z, 0.3) -- (\z, -0.3) node (\z) [above=1.4cm] { } node [below] {\w};
\foreach \z\w in {9/y} \draw (\z, 0.3) -- (\z, -0.3) node (\z) [above=1.4cm] { } node [below=0.065cm] {\w};
\node at (8,0.7)  [fill,circle,inner sep=0.15cm,color=blue] { };
\node at (8,1.2) [fill,circle,inner sep=0.05cm,color=black] { };
\node at (8,1.4) [fill,circle,inner sep=0.05cm,color=black] { };
\node at (8,1.6) [fill,circle,inner sep=0.05cm,color=black] { };
\node at (8,2.1) [fill,circle,inner sep=0.15cm,color=blue] { };
\node (c) at (9,0.7) [fill,circle,inner sep=0.15cm,color=blue] { };
\node (c) at (9,1.3) [fill,circle,inner sep=0.15cm,color=blue] { };
\node (c) at (9,1.8) [fill,circle,inner sep=0.05cm,color=black] { };
\node (c) at (9,2) [fill,circle,inner sep=0.05cm,color=black] { };
\node (c) at (9,2.2) [fill,circle,inner sep=0.05cm,color=black] { };
\node (c1) at (9,2.7) [fill,circle,inner sep=0.15cm,color=blue] { };
\node at (9,2.8) (Z) [above] { };
\node at (8,2.8) (W) [above] { };
\node at (10,2.8) (Y) [above] { };
\draw[-latex] (Z) to[out=130,in=80] node[midway,font=\scriptsize,above]  { } (W);
\draw [decorate,decoration={brace,amplitude=0.1cm, mirror}] (6.5,0.4) -- (6.5,3) node [black,midway,xshift=0.3cm] {\scriptsize$\alpha$};
\node at ([shift={(40:-2)}]9.4,4.8) {\scriptsize $\hspace{-0.2cm}\eta(y)[\alpha - \eta(y-1)]$};
\draw[-latex] (Z) to[out=60,in=120] node[midway,font=\scriptsize,above] {\scriptsize $\hspace{1.3cm}\eta(y)[\alpha-\eta(y+1)]$} (Y);
\foreach \x\y in {12/N-2,13/N-1,14/N} \draw (\x, 0.3) -- (\x, -0.3) node (\x) [above=1cm] { } node [below] {\y};
\node at (11,-0.5) {\ldots};
\node at (13,0) (B') [below = 1.2cm] { };
\node at (14,0) (B) [below = 1.8cm] { };
\node at (14,0) [below = 0.8cm] {right};
\node at (14,0) [below = 1.2cm] {reservoir};
\draw[-latex] (13) to[out=130,in=60] node[midway,font=\scriptsize,above] {$\hspace{-1.8cm}\eta(N-1)[\alpha-\eta(N-2)]$} (12);
\draw[-latex,color = red] (13) to[out=80,in=130] node[midway,font=\scriptsize,above,color=red] {$\hspace{0.5cm}\frac{\beta}{N^\theta} \eta(N-1)$} (14);
\draw[-latex, color = blue] (B) to [bend left] node[midway,font=\scriptsize,below=0.1cm,color=blue] {$\frac{\delta}{N^\theta} [\alpha-\eta(N-1)]$} (B');
\end{tikzpicture}
\end{figure}

The SEP($\alpha$) describes an irreducible continuous-time Markov chain with a finite state-space and so it admits a unique invariant measure. In equilibrium, i.e. when $\frac{\epsilon}{\epsilon+\gamma}= \frac{\delta}{\delta + \beta}$, the invariant measure is reversible and explicitly known. For  a function $\varrho:[0,1]\to[0,1]$, let  $\nu^N_{\varrho(\cdot)}$ be the product measure whose marginals are given by the  Binomial($\alpha$, $\varrho(\tfrac xN)$) distribution, i.e.
\begin{align} \label{reference_measure}
\nu^N_{\varrho(\cdot)}(\eta) = \prod_{x \in\Lambda_N} {\alpha \choose \eta(x)} \left[\varrho \left(\tfrac{x}{N} \right) \right]^{\eta(x)} \left[1-\varrho \left(\tfrac{x}{N} \right) \right]^{\alpha - \eta(x)}.
\end{align}
\begin{proposition}[\cite{carinci2013duality}]
Let $\alpha \in \mathbb{N}$ and $\varrho \in (0,1)$. If 
$ \frac{\epsilon}{\epsilon+\gamma} =\frac{\delta}{\delta + \beta}= \varrho,
$ then the homogeneous product measure $\nu^N_{\varrho}(\cdot)$ given in \eqref{reference_measure} with $\varrho(\cdot)\equiv \varrho$
is reversible.
\end{proposition}
Hereafter we fix $T>0$ and a finite time horizon $[0,T]$.  We denote by $\mcb D([0,T],\Omega_N)$ as the space of càdlàg trajectories endowed with the Skorohod topology and $\mcb{M}$ as the space of non-negative Radon measures on $[0,1]$ with total mass bounded by $\alpha$ and equipped with the weak topology. For   $\eta \in \Omega_N$, we define the empirical measure $\pi^{N}(\eta,du)$ by 
\begin{equation*}\label{MedEmp}
\pi^{N}(\eta, du):=\dfrac{1}{N-1}\sum _{x\in\Lambda_N }\eta(x)\delta_{\frac{x}{N}}\left( du\right) \in \mcb{M},
 \end{equation*}
where $\delta_{b}$ is a Dirac measure in $b \in [0,1]$. For every $G: [0,1] \rightarrow \mathbb{R}$, we denote the integral of $G$ with respect to $\pi^N$ by $\langle \pi^N, G \rangle$
and we observe that 
$$ \langle \pi^N, G \rangle := \frac{1}{N-1} \sum_{x\in\Lambda_N } \eta(x) G \left( \tfrac{x}{N}  \right).
$$
We also define $\pi^{N}_{t}(\eta, du):=\pi^{N}(\eta_{tN^2}, du) $.

\begin{definition}
Let $\mathfrak g: [0,1]\rightarrow [0,\alpha]$ be a measurable function. We say that a sequence  of probability measures {$(\mu_N)_{N \geq 1}$} on $\Omega_N$  is {associated to the profile $\mathfrak g$} if for every continuous  function $G$ and $\delta >0$, it holds
$$
\lim_{N \rightarrow \infty} \mu_N \Big( \eta \in \Omega_N: \big| \langle \pi^N, G \rangle - \int_0^1 G(u) \mathfrak g (u) du \big| > \delta \Big) =0. 
$$
\end{definition}
For every $N \geq 1$, let $\mathbb{P} _{\mu_{N}}$ be the probability measure on $\mcb D([0,T],\Omega_N)$ induced by the Markov process \pat{$(\eta_{tN^2})_{t\geq 0}$} and by the initial  measure $\mu_{N}$ and the expectation with respect to $\mathbb{P}_{\mu_{N}}$ is denoted  by $\mathbb{E}_{\mu_{N}}$.  We denote by $\mcb D([0,T], \mcb{M})$ the space of càdlàg trajectories endowed with the Skorohod topology and   $(\mathbb{Q}_{N})_{n \geq 1}$ as the sequence of probability measures on $\mcb D([0,T],\mcb M)$ induced by the Markov process $(\pi_{t}^{N})_{0 \leq t \leq T}$ and by the initial measure  $\mu_{N}$.

\subsection{Hydrodynamic equations}

In order to properly introduce our notions of solutions, which are in the weak sense, we first need to define the set of test functions. For $m,n\in\mathbb N_0$,  let $C^{m,n}([0, T] \times [0,1])$ be the set of continuous functions defined on $[0, T] \times [0,1] $ that are $m$ times differentiable on the first variable and $n$ times differentiable on the second variable, and with continuous derivatives. We also denote  $C^{m,n}_c ([0,T] \times [0,1])$ as the set of  functions $G \in C^{m,n}([0, T] \times[0, 1])$ such that for each time $s$,  $G_s$ has a compact support included in $(0,1)$ and we denote by $C_{c}^{m}(0,1)$ (resp. $C_c^\infty (0,1)$) the set of all $m$ continuously differentiable (resp. smooth) real-valued  functions defined on $(0,1)$ with compact support. The supremum norm is denoted by $\| \cdot \|_{\infty}$. Now we define the Sobolev space $\mcb{H}^1$ on $[0,1]$. For that purpose, we define the semi inner-product $\langle \cdot, \cdot \rangle_{1}$ on the set $C^{\infty} ([0,1])$ by 
$\langle G, H \rangle_{1} = \langle\partial_u G,\partial_u H\rangle
$
for $G,H\in C^{\infty} ([0,1])$
and  the corresponding semi-norm is denoted by $\| \cdot \|_{1}$.  Above $\langle\cdot,\cdot\rangle$ corresponds to the inner product in $\mathbb L^2([0,1])$  and should not be mistaken with $\langle \pi^N, G \rangle$.  The corresponding norm is denoted by $\| \cdot \|_{L^2}^2 .$

\begin{definition}
\label{Def. Sobolev space}
The Sobolev space $\mcb{H}^{1}$ on $[0,1]$ is the Hilbert space defined as the completion of $C^\infty ([0,1])$ for the norm 
$$\| \cdot\|_{{\mcb{H}}^1}^2 :=  \| \cdot \|_{L^2}^2  +  \| \cdot \|^2_{1}$$ and we note that its elements coincide a.e. with continuous functions. The space $L^{2}(0,T;\mcb{H}^{1})$ is the set of measurable functions $f:[0,T]\rightarrow  \mcb{H}^{1}$ such that 
$\int^{T}_{0} \Vert f_{s} \Vert^{2}_{\mcb{H}^{1}}ds< \infty.$
\end{definition}
Let  $\mathfrak g:[0,1]\rightarrow [0,\alpha]$ be a measurable function which will be the initial condition  in  all our equations.  We use the notation
\pat{$$
\rho_-:=\alpha \frac{\epsilon}{\epsilon+\gamma}\quad \textrm{and}\quad  
\rho_+:=\alpha \frac{\delta}{\delta + \beta}
$$
to identify, respectively,} the left and right density of the boundary reservoirs.
\begin{definition}
\label{Def. Dirichlet source Condition-g_ssep}
We say that  $\rho:[0,T]\times[0,1] \to [0,\alpha]$ is a weak solution of the heat equation with Dirichlet boundary conditions
 \begin{equation}\label{eq:Dirichlet source Equation-g_ssep}
 \begin{cases}
 &\partial_{t}\rho_{t}(u)=\alpha\Delta\, {\rho} _{t}(u), \quad (t,u) \in [0,T]\times(0,1),\\
 &{ \rho} _{t}(0)=\rho_-, \quad { \rho}_{t}(1)=\rho_+,\quad t \in (0,T],
 \end{cases}
\end{equation}
if $\rho \in L^{2}(0,T;\mcb{H}^{1})$, ${ \rho} _{t}(0)=\rho_-$ and ${ \rho}_{t}(1)=\rho_+$ for a.e. $t \in (0,T]$, and  for all $t\in [0,T]$ and any $G \in C_c^{1,2} ([0,T]\times[0,1])$ it holds
$$\langle \rho_{t},  G_{t}\rangle  -\langle \mathfrak g,   G_{0} \rangle
- \int_0^t\langle \rho_{s},\Big(\alpha\Delta + \partial_s\Big) G_{s}\rangle ds=0.
$$
\end{definition}
\begin{definition}
\label{Def. Robin Condition-g_ssep} Let $\kappa\geq0$. 
We say that $\rho:[0,T]\times[0,1] \to [0,\alpha]$ is a weak solution of the heat equation with Robin  boundary conditions 
 \begin{equation}\label{Robin Equation-g_ssep}
 \begin{cases}
 &\partial_{t}\rho_{t}(u)=\alpha\Delta\, {\rho} _{t}(u), \quad (t,u) \in [0,T]\times(0,1),\\ &\partial_{u}\rho _{t}(0)=\kappa \frac{\epsilon + \gamma}{\alpha}\Big(\rho_{t}(0) -\rho_-\Big),\quad \partial_{u} \varrho_{t}(1)= \kappa \frac{\delta + \beta}{\alpha} \Big(\rho_+-\rho_t(1)\Big),\quad t \in (0,T],
 \end{cases}
 \end{equation}
if $\rho \in L^{2}(0,T;\mcb{H}^{1})$ and for all $t\in [0,T]$ and any $G \in C^{1,2} ([0,T]\times[0,1])$ it holds
\begin{equation*}\label{eq:Robin integral-g_ssep}
\begin{split}
\langle \rho_{t},  G_{t} \rangle -\langle \mathfrak g,   G_{0} &\rangle   - \int_0^t\langle\rho_{s}, \Big(\alpha\Delta + \partial_s\Big) G_{s} \rangle ds + \alpha \int^{t}_{0} {\left[   \rho_{s}(1)\partial_uG_{s}(1)-\rho_{s}(0) \partial_u G_{s}(0)  \right] } \, ds\\
& \qquad-\kappa \int^{t}_{0}  \left[ G_{s}(0)(\epsilon + \gamma)\Big(\rho_{s}(0) -\rho_-\Big) +G_{s}(1) (\delta + \beta) \Big(\rho_+ - \rho_s(1)\Big)  \right] \,  ds=0.
\end{split}   
\end{equation*}
\end{definition}
Taking $\kappa=0$ in Definition \ref{eq:Robin integral-g_ssep}, we get the heat equation with Neumann boundary conditions.
\begin{remark}
We observe that the uniqueness of the weak solutions as given above can be seen, for example, in \cite{baldasso2017exclusion,byronMPRF} and we refer the interested reader to those articles for a proof. 
\end{remark}
We observe that the solution of the hydrodynamic equation $\rho_t(u)$ takes values in $[0,\alpha]$ and should not be confused with the function \pat{$\varrho(\cdot)$, defined in  \eqref{reference_measure},  that takes values in $[0,1]$.} 
\subsection{Hydrodynamic limit} Now we state our main result. 
\begin{theorem}
\label{th:hyd_ssep}
Let $\mathfrak g:[0,1]\rightarrow[0,\alpha]$ be a measurable function and let \pat{$(\mu _{N})_{N\geq 1}$} be a sequence of probability measures  associated to $\mathfrak g$. For any $t\in[0,T]$ and any $G\in  C^0([0,1])$, it holds
\begin{equation*}\label{limHidreform}
 \lim _{N\to\infty } \mathbb P_{\mu _{N}}\big( \eta_{\cdot} : \Big|\dfrac{1}{N-1}\sum_{x \in \Lambda_{N} }G\left(\tfrac{x}{N} \right)\eta_{tN^2}(x) - \langle G,\rho_{t}\rangle\Big|    > \delta \Big)= 0,
\end{equation*}
where  $\rho_{t}(\cdot)$ is the unique weak solution of : 

a)  \eqref{eq:Dirichlet source Equation-g_ssep},  if $\theta<1$;

b)  (\ref{Robin Equation-g_ssep}) with $\kappa=1$ if $\theta =1$, and with $\kappa=0$ if $\theta>1$.
\end{theorem}
\section{Proof of Theorem \ref{th:hyd_ssep}} \label{section3}
The strategy of the proof follows from the entropy method introduced in \cite{GPV} and it consists in two ingredients. First, we show in Subsection  \ref{sec:tightness} that the sequence $(\mathbb{Q}_{N})_{n \geq 1}$ is tight with respect to the Skorohod topology in 
$\mcb D([0,T], \mcb M)$. From this  we know that there exists a limit point $\mathbb Q^*$ that we want to characterize uniquely, so that convergence of the whole sequence follows. Second, in Subsection \ref{sec:charac}, we characterize the limit points. 
Since $\eta_t(x)$ is bounded, the limit point $\mathbb Q^*$ is supported on trajectories of measures that are absolutely continuous with respect to Lebesgue, i.e. $\pi_t(du)=\rho_t(u)du$. Then we show that the density satisfies the integral formulation of the corresponding equations by using Dynkin's formula and the Replacement Lemmas of Section \ref{section4}. In order to show  that $  \rho \in L^2(0,T;\mcb H^1)$, we refer to \cite{baldasso2017exclusion} since, once we have the replacement lemma given in Lemma \ref{repl_lemma_ave}, the proof is similar to the one in \cite{baldasso2017exclusion}.  Finally, the convergence follows by the uniqueness of the weak solution of the hydrodynamic equation. We observe however that in the Dirichlet case we also have to prove that ${ \rho} _{t}(0)=\rho_-$ and ${ \rho}_{t}(1)=\rho_+$ for a.e. $t\in(0,T]$ and this follows easily from the combination of  Lemmas \ref{rep_lemma} and \ref{repl_lemma_ave}, for details we refer the reader to Section 5.3 of \cite{byronMPRF}.

\subsection{Tightness}
\label{sec:tightness}
This section is devoted to the proof of tightness of the sequence $(\mathbb{Q}_{N})_{ N \geq 1 }$. From  Proposition 4.1.6  of \cite{kipnis1998scaling},  it is enough to show that, for every $\varepsilon >0$ and any function $G\in C^0_c([0,1])$,
\begin{equation} \label{T1sdif}
\displaystyle \lim _{\delta \rightarrow 0^+} \limsup_{N \rightarrow\infty} \sup_{\tau  \in \mcb{T}_{T},\bar\tau \leq \delta} {\mathbb{P}}_{\mu _{N}}\Big[\eta_{\cdot}\in \mcb D ( [0,T], \Omega_N) :\Big|\langle\pi^{N}_{\tau+ \bar\tau},G\rangle-\langle\pi^{N}_{\tau},G\rangle\Big| > \epsilon \Big]  =0, 
\end{equation}
where $\mcb{T}_T$ represents  the set of stopping times bounded by $T$.
By an $\mathbb L^1$ approximation procedure,  it is enough to show the
last result  for $G \in C_{c}^2([0,1])$. From Dynkin's formula, see, for example  Appendix 1 of \cite{kipnis1998scaling},  for every $t \geq 0$ and $G$ sufficiently smooth
\begin{equation} \label{defMnt}
\mcb M_{t}^{N}(G) = \langle \pi_{t}^{N},G \rangle - \langle \pi_{0}^{N},G  \rangle  - \int_0^t N^2 \mcb {L}_N \langle  \pi_{s}^{N},G  \rangle ds,
\end{equation}
is a martingale.  We observe that \eqref{T1sdif} is a direct consequence of  the next result combined with  Markov's inequality.
\begin{proposition} \label{tightcond1dif}
For $G \in C_{c}^2([0,1])$, it holds
\begin{align*}
\lim_{\delta \rightarrow 0^+} \limsup_{N \rightarrow \infty} \sup_{\tau \in \mcb{T}_T, \bar{\tau} \leq \delta} \mathbb{E}_{\mu_N} \left[ \Big| \int_{\tau}^{\tau+ \bar{\tau}} N^2  \mcb L_{N}\langle \pi_{s}^{N},G\rangle ds \Big| \right] = 0.\\
\lim_{\delta \rightarrow 0^+} \limsup_{N \rightarrow \infty} \sup_{\tau \in \mcb{T}_T, \bar{\tau} \leq \delta} \mathbb{E}_{\mu_N} \left[ \left( \mcb M_{\tau}^{N}(G) -  \mcb M_{\tau+\bar{\tau}}^{N}(G) \right)^2 \right] = 0.
\end{align*}
\end{proposition}
\begin{proof}
The proof of the first limit  follows by proving that 
$N^2  |\mcb{L}_N \langle \pi_{s}, G \rangle| \leq C$. 
    A simple, but long, computation shows that
    \begin{align*}
        N^2 \mcb{L}_N \langle& \pi^N_s, G \rangle = \alpha  \frac{1}{N-1} \sum_{x=1}^{N-1} \eta_{sN^2}(x) \Delta_N G\left(\tfrac{x}{N} \right) \\&+ \frac{\epsilon + \gamma}{(N-1)N^\theta} N^2 \left[\rho_- - \eta_{sN^2}(1) \right] G\left(\tfrac{1}{N} \right) +\alpha \frac{N^2}{N-1} \eta_{sN^2}(1) \left( G\left(\tfrac{1}{N}\right) - G\left(0\right) \right) \\
        &+ \frac{\delta + \beta}{(N-1)N^\theta} N^2 \left[ \rho_+- \eta_{sN^2}(N-1) \right] G\left(\tfrac{N-1}{N} \right)+\alpha \frac{N^2}{N-1} \eta_{sN^2}(N-1) \left( G\left(\tfrac{N-1}{N} \right) - G\left(1 \right)\right). 
    \end{align*}
    From last identity together with the fact that $G \in C_{c}^2([0,1])$ and $|\eta_s(x)|\leq \alpha$ for all $x\in\Lambda_N$ and $s\in[0,T]$, the result easily follows. 
To prove the second limit we note that
$$
\mathbb{E}_{\mu_N} \left[ \left( \mcb M_{\tau}^{N}(G) -  \mcb M_{\tau+\bar{\tau}}^{N}(G) \right)^2 \right] = \mathbb{E}_{\mu_N} \Big[ \int_{\tau}^{\tau+\bar{\tau}}N^2 \left(\mcb L_{N}[\langle \pi_{s}^{N},G \rangle ^{2}]-2\langle \pi_{s}^{N},G \rangle \mcb L_{N} \langle \pi_{s}^{N},G \rangle\right) ds \Big].
$$
Above we used the fact that
\begin{equation}
\mcb{N}^N_t(G) = \left[\mcb M^N_t(G)\right]^2 - \int_0^t \left[N^2\mcb L_{N} \langle \pi_{s}^{N},G \rangle^2 - \langle \pi_{s}^{N},G \rangle N^2\mcb L_{N} \langle \pi_{s}^{N},G \rangle \right] ds
\end{equation} 
is a martingale (see, e.g.,  Appendix 1 of \cite{kipnis1998scaling}).
A simple, but long, computation  shows that, for $\eta \in \Omega_N$, it holds
\begin{align*}
        &N^2 [\mcb{L}_N\langle \pi^N, G \rangle^2 - 2 \langle \pi^N, G \rangle \mcb{L}_N\langle \pi^N, G \rangle] \\
        &= \frac{N^2}{(N-1)^3} \sum_{x=1}^{N-2} \Big\{ \eta(x)[\alpha - \eta(x+1)] + \eta(x+1)[\alpha - \eta(x)]  \Big\} \left[ G \left( \tfrac{x+1}{N}\right) -  G \left( \tfrac{x}{N}\right)\right]^2 \\
        &+ \frac{N^2}{N^\theta (N-1)^2} \Big[ \left(\gamma \eta(1) + \epsilon [\alpha-\eta(1)] \right) G^2 \left(\tfrac{1}{N}\right) + \left(\beta \eta(N-1) + \delta [\alpha-\eta(N-1)] \right) G^2 \left(\tfrac{N-1}{N} \right) \Big].
\end{align*}
Since $G \in C_{c}^2([0,1])$, last display vanishes as $N\to+\infty$,   from where the proof ends. 
       \end{proof}
\subsection{Characterization of limit points}\label{sec:charac}
In this subsection we explain how to get the integral notions of weak solutions from the microscopic system. Our starting point is Dynkin's formula which tells us that for  $G$  regular enough
\begin{align*}
    &\mcb{M}^N_t(G) := \langle \pi^N_t, G_t \rangle - \langle \pi^N_0, G_0 \rangle - \int_0^t \langle \pi^N_s, \partial_s G_s \rangle + \alpha \langle \pi^N_s, \Delta_N G_s \rangle ds \\
    &- \int_0^t \left\{\frac{\epsilon + \gamma}{(N-1)N^\theta} N^2 \left[\rho_-- \eta_{sN^2}(1) \right] G_s\left(\tfrac{1}{N} \right) +\alpha \frac{N}{N-1} \eta_{sN^2}(1) \nabla^+_N  G_s\left(0\right)  \right\} ds \\
    &+ \int_0^t \left\{ \alpha \frac{N}{N-1} \eta_{sN^2}(N-1) \nabla^+_N  G_s\left(\tfrac{N-1}{N} \right) - \frac{\delta + \beta}{(N-1)N^\theta} N^2 \left[ \rho_+- \eta_{sN^2}(N-1) \right] G_s\left(\tfrac{N-1}{N} \right) \right\} ds
\end{align*}
is a martingale. Above we used the usual notations for the discrete laplacian and derivative, i.e.
\begin{equation*}
    \Delta_N G\left(\tfrac{x}{N} \right) := N^2\left[ G\left(\tfrac{x-1}{N} \right) - 2 G\left(\tfrac{x}{N} \right) + G\left(\tfrac{x+1}{N} \right) \right]\quad 
\textrm{and}\quad
    \nabla^+_N G\left( \tfrac{x}{N} \right) := N \left[ G\left(\tfrac{x+1}{N} \right) - G\left(\tfrac{x}{N} \right)\right].
\end{equation*}
 From the computations of the previous proof we know that the martingale $\mcb M^N_t (G)$ vanishes in $\mathbb L^2(\mathbb P_{\mu_N})$ as $N\to+\infty$. It remains to analyse the other terms. From here on we take into account the value of $\theta.$ We present the argument for $\theta~\geq 1 $, since for $\theta < 1$ the test function $G \in C_c^{1,2}([0,T]\times[0,1])$  and the argument is much simpler.  Now,  for  $\theta
\geq 1$ we assume that $G \in C^{1,2}([0,T]\times[0,1])$. We start with $\theta>1$  and we see that
 \begin{equation*}\begin{split}
    \mcb{M}^N_t(G) =& \langle \pi^N_t, G_t \rangle - \langle \pi^N_0, G_0 \rangle - \int_0^t \langle \pi^N_s, (\partial_t G_s + \alpha \Delta_N G_s )\rangle ds \\
    -& \int_0^t \left\{\alpha  \eta_{sN^2}(1) \nabla^+_N  G_s\left(0\right) -\alpha  \eta_{sN^2}(N-1) \nabla^+_N  G_s\left(\tfrac{N-1}{N} \right)  \right\} ds + O(N^{1-\theta})
\end{split}\end{equation*}
is a martingale.
Since $G$ is of class $C^2([0,1])$, from Lemma \ref{repl_lemma_ave}, last identity is equivalent to
 \begin{equation*}\begin{split}
    \mcb{M}^N_t(G) =& \langle \pi^N_t, G_t \rangle - \langle \pi^N_0, G_0 \rangle - \int_0^t \langle \pi^N_s, (\partial_t G_s +\alpha \Delta_N G_s) \rangle ds \\
    -& \int_0^t \left\{\alpha  \eta_{sN^2}^{\epsilon N}(1) \nabla^+_N  G_s\left(0\right) -\alpha  \eta_{sN^2}^{\epsilon N}(N-1) \nabla^+_N  G_s\left(\tfrac{N-1}{N} \right)  \right\} ds + O(N^{1-\theta})+o(1),
\end{split}\end{equation*}
plus terms whose $\mathbb L^1(\mathbb P_{\mu_N})$-norm vanishes as $N\to+\infty$ and then $\epsilon \to 0$. Now fix $\theta=1$. Again from Lemma \ref{repl_lemma_ave} we get that 
\begin{align*}
    \mcb{M}^N_t(G) &= \langle \pi^N_t, G_t \rangle - \langle \pi^N_0, G_0 \rangle - \int_0^t \langle \pi^N_s, ( \partial_t G_s + \alpha\Delta_N G_s) \rangle ds \\
    &- \int_0^t \left\{(\epsilon + \gamma) \left[\rho_-- \eta_{sN^2}^{\epsilon N}(1) \right] G_s\left(\tfrac{1}{N} \right) +\alpha  \eta_{sN^2}^{\epsilon N}(1) \nabla^+_N  G_s\left(0\right)  \right\} ds \\
    &+ \int_0^t \left\{ \alpha \eta_{sN^2}^{\epsilon N}(N-1) \nabla^+_N  G_s\left(\tfrac{N-1}{N} \right) - (\delta + \beta) \left[ \rho_+- \eta_{sN^2}^{\epsilon N}(N-1) \right] G_s\left(\tfrac{N-1}{N} \right) \right\} ds+o(1).
\end{align*}
Now we can conclude the argument in  all cases. To recognize the limiting equations it is enough to observe the following facts. First,  note that for any $x\in\Lambda_N$ it holds  $\eta_t^{\epsilon N}(x)=\langle \pi^N_t, \iota_\epsilon^{x/N}\rangle$, where for $v\in[0,1]$, $\iota_\epsilon^{v}(u)=\epsilon^{-1}\mathbb {1}_{(v,v+\epsilon]}(u)$. From tightness, we know the convergence of a subsequence $\pi_t^{N_k}(\eta, du)$ to $\pi_t(du)$, given  by $\pi_t(du)=\rho_t(u)du$. 
Moreover, since $\rho_t(u)\in[0,\alpha]$ for all $t\in[0,T]$ and $u\in[0,1]$, from Lebesgue's differentiation theorem we can conclude that 
$$\lim_{\epsilon\to 0}\Big|\rho_s(u)-\tfrac{1}{\epsilon}\int_{u}^{u+\epsilon} \rho_s(v)dv\Big| =0$$
for almost every $u\in[0,1]$. But to recognize the boundary conditions we need the last result to be true for the boundary points $u=0$ and $u=1$.  This is easy and it is analogous to the proof of Lemma 6.2 in \cite{BDGN}. We leave the details to the reader. 
\section{Replacement lemmas} \label{section4}
In this section we {collect} several results that were used along the proofs. From here on we assume that  $\varrho(\cdot)$ is a profile  bounded away from zero and one, i.e. there exist $0<a< b<1 $ such that, for all $x\in\Lambda_N$, it holds 
$$
    0 < a \leq  \varrho\left( \tfrac{x}{N}\right) \leq b < 1.
$$
We claim that under the previous conditions,  for any probability measure $\mu_N$,  the entropy of $\mu_N$ with respect to \pat{the Binomial  product measure of parameters ($\alpha$, $\varrho(\tfrac xN)$)} , i.e. the measure $\nu^N_{\varrho(\cdot)}$ given in \eqref{reference_measure}, satisfies $$
    H(\mu_N | \nu^N_{\varrho(\cdot)}) \lesssim N. $$
To prove the claim, observe that by the definition of the entropy, since $\mu_N$ is a probability measure, we have  
\begin{align*}
    H\left(\mu_N  | \nu^N_{\varrho(\cdot)} \right) &= \sum_{\eta \in \Omega_N} \mu_N(\eta) \log \left( \frac{\mu_N(\eta)}{\nu^N_{\varrho(\cdot)}(\eta)} \right) \nu^N_{\varrho(\cdot)}(\eta)  \leq \max_{\eta \in \Omega_N} \left\{ \log \left[\nu^N_{\varrho(\cdot)}(\eta)\right]^{-1} \right\}.
\end{align*}
For $\eta \in \Omega_N$, recalling the expression of $\nu^N_{\varrho(\cdot)}$ given in \eqref{reference_measure}, we have 
\begin{align*}
    \log \left[\nu^N_{\varrho(\cdot)}(\eta)\right]^{-1} &= \sum_{x=1}^{N-1} \log \left[ \frac{\eta(x)! [\alpha - \eta(x)]!}{\alpha!} \right] + \eta(x) \log \left[ \varrho \left( \tfrac{x}{N}\right)^{- 1} \right] + [\alpha - \eta(x)]\log \left[1 - \varrho \left( \tfrac{x}{N}\right)\right]^{-1}\\
    &\leq (N-1)\left\{ \log \left[ \alpha! \right] + \alpha \log a^{- 1} + \alpha \log \left[1 - b\right]^{-1} \right\}.
\end{align*} 
\subsection{Estimating  Dirichlet forms}\label{sec:dir_form}
In this section we want to compare the Dirichlet form of our model with a quadratic form, \pat{given in terms of the carré du champ operator}. To properly do that we take as reference measure $\nu^N_{\varrho(\cdot)}$ given in  \eqref{reference_measure} and we consider a Lipschitz continuous profile $\varrho(\cdot)$ which is bounded away from zero and one and locally constant at the boundary satisfying $\varrho(0)=\frac{\epsilon}{\epsilon+\gamma}$ and $\varrho(1)=\frac{\delta}{\delta + \beta}$.
\begin{lemma}\label{lem:Dir}
For any density $f$, if $\varrho(\cdot)$ satisfies
 the previous conditions, then there exist constants $C, K >0$ and $N_0\in\mathbb N$, such that for $a\in\{\ell, r\}$, it holds
\pat{$$ \langle \mcb {L}_{a} \sqrt{f}, \sqrt{f} \rangle_{\nu^N_{\varrho(\cdot)}} =  - \frac{1}{2} D^{a}_{\nu^N_{\varrho(\cdot)}}(\sqrt{f}),\forall  N\geq N_0$$} and 
\begin{align*} \langle  \mcb {L}_{bulk} \sqrt{f}, \sqrt{f} \rangle_{\nu^N_{\varrho(\cdot)}} \leq - K  D^{bulk}_{\nu^N_{\varrho(\cdot)}}(\sqrt{f}) + \frac{C}{N^2}, ~\forall N\geq 1,
\end{align*}
where
        \begin{align*} 
        D_{\nu^N_{\varrho(\cdot)}}^\ell(\sqrt{f}) &:=\frac{1}{N^\theta} \int_{\Omega_N} \!\left[ c_{1,0}(\eta)\Big\{\!\sqrt{f}(\eta^{1,0}) - \!\sqrt{f}(\eta)\!\Big\}^2 + c_{0,1}(\eta)\Big\{\sqrt{f}(\eta^{0,1}) - \sqrt{f}(\eta)\Big\}^2 \right]d\nu^N_{\varrho(\cdot)} \\
        D^{bulk}_{\nu^N_{\varrho(\cdot)}}(\sqrt{f}) &:=\sum_{x = 1}^{N-2} D_ {\nu^N_{\varrho(\cdot)}}^{x,x+1}(\sqrt{f}) +D_{\nu^N_{\varrho(\cdot)} }^{x+1,x}(\sqrt{f} ) \\&= \sum_{x = 1}^{N-2} \Big\{\int_{\Omega_N} \! c_{x,x+1}(\eta) \Big\{ \!\sqrt{f}(\eta^{x,x+1}) - \! \sqrt{f}(\eta) \! \Big\}^2 d\nu^N_{\varrho(\cdot)}\\&\hspace*{1cm} +\int_{\Omega_N} \! c_{x+1,x}(\eta) \Big\{ \! \sqrt{f}(\eta^{x+1,x}) - \! \sqrt{f}(\eta) \! \Big\}^2 d\nu^N_{\varrho(\cdot)}\Big\}
    \end{align*}
    and the definition of  $D_{\nu^N_{\varrho(\cdot)}}^r(\sqrt{f})$  is analogous to the one of  $ D_{\nu^N_{\varrho(\cdot)}}^\ell(\sqrt{f}) $ by replacing $0$ and $1$ by $N$ and $N-1$, respectively.
 \end{lemma}
\begin{proof}
We present the proof in details for the left boundary (for the right one it is analogous) and the bulk generators. \pat{For the left boundary generator we have}
\begin{flalign} \label{eq:bound_1}
    \langle \mcb{L}_{\ell} \sqrt{f}, \sqrt{f} \rangle_{\nu^N_{\varrho(\cdot)}} &= - \tfrac{1}{2} D_{\nu^N_{\varrho(\cdot)}}^{\ell}(\sqrt{f})\nonumber\\ +& \frac{1}{2N^\theta}\int_{\Omega_N} \left[ c_{1,0}(\eta) \Big\{f(\eta^{1,0}) - f(\eta)\Big\} + c_{0,1}(\eta)\Big\{f(\eta^{0,1}) - f(\eta)\Big\}\right] d \nu^N_{\varrho(\cdot)}
\end{flalign}
From a change of  variables  and the fact that
\begin{align*}
     \nu^N_{\varrho(\cdot)}(\eta^{1,0}) =  \frac{\eta(1)}{\alpha - \eta^{1,0}(1)} \frac{1-\varrho(\frac{1}{N})}{\varrho(\frac{1}{N})} \nu^N_{\varrho(\cdot)}(\eta)\quad \textrm{and}\quad 
    \nu^N_{\varrho(\cdot)}(\eta^{0,1}) = \frac{\alpha -\eta(1)}{\eta^{0,1}(1)} \frac{\varrho(\frac{1}{N})}{1-\varrho(\frac{1}{N})} \nu^N_{\varrho(\cdot)}(\eta),
\end{align*} we conclude that  \eqref{eq:bound_1} is equal to 
\begin{align} \nonumber
 \frac{\gamma + \epsilon}{2 N^\theta} \int_{\Omega_N} \left( \frac{\eta(1)}{\varrho(\frac{1}{N})} - \frac{\alpha-\eta(1)}{1 - \varrho(\frac{1}{N})} \right) f(\eta) \left[\frac{\epsilon}{\epsilon+\gamma} - \varrho \left(\tfrac{1}{N} \right) \right] d \nu^N_{\varrho(\cdot)}.&
\end{align}
  Observe that since $\varrho(0)=\frac{\epsilon}{\epsilon+\gamma}$ and the profile is locally constant, last display vanishes for $N$ sufficiently big.
Finally we treat the bulk generator.  Observe that by summing and subtracting proper terms, we get
\begin{align}
    \langle \mcb{L}_{bulk} \sqrt{f}, &\sqrt{f} \rangle_{\nu^N_{\varrho(\cdot)}} =\\ &- \frac{1}{4}\sum_{x = 1}^{N-2}  [D_{\nu^N_{\varrho(\cdot)}}^{x,x+1}(\sqrt{f}) +  D_{\nu^N_{\varrho(\cdot)}}^{x+1,x}(\sqrt{f})] \nonumber &\\
    &+ \frac{1}{4}\sum_{x = 1}^{N-2} \int_{\Omega_N} \left\{c_{x,x+1}(\eta) [f(\eta^{x,x+1}) - f(\eta)]+ c_{x+1,x}(\eta) [f(\eta^{x+1,x}) - f(\eta)] \right\}d\nu^N_{\varrho(\cdot)}\nonumber &\\
    &+ \frac{1}{2}\sum_{x = 1}^{N-2} \int_{\Omega_N} c_{x,x+1}(\eta) \Big\{ \sqrt{f}(\eta^{x,x+1}) - \sqrt{f}(\eta) \Big\}\sqrt{f}(\eta)d\nu^N_{\rho(\cdot)}\label{I} &\\
    &+ \frac{1}{2}\sum_{x = 1}^{N-2} \int_{\Omega_N} c_{x+1,x}(\eta) \Big\{ \sqrt{f}(\eta^{x+1,x}) - \sqrt{f}(\eta) \Big\}\sqrt{f}(\eta) d\nu^N_{\varrho(\cdot)}\label{II}.
\end{align}
Now we treat the terms \eqref{I} and \eqref{II}. We first  make a change of variables,  i.e. $\xi = \eta^{x,x+1}$ and $\xi = \eta^{x+1,x}$ in \eqref{I} and \eqref{II}, respectively. Then, we sum and subtract appropriate terms, use the fact that 
\begin{equation}\label{eq:change_variables}\begin{split} c_{x+1,x}(\eta^{x,x+1}) \nu^N_{\varrho(\cdot)}(\eta^{x,x+1}) &= \frac{1}{a_x} c_{x,x+1}(\eta) \nu^N_{\varrho(\cdot)}(\eta), 
\\ 
c_{x,x+1}(\eta^{x+1,x}) \nu^N_{\varrho(\cdot)}(\eta^{x+1,x}) &= a_x c_{x+1,x}(\eta) \nu^N_{\varrho(\cdot)}(\eta),
\end{split}\end{equation}
where 
\begin{equation}\label{eq:a_x} a_x:= \frac{\varrho(\frac{x}{N})(1-\varrho(\frac{x+1}{N}))}{\varrho(\frac{x+1}{N})(1-\varrho(\frac{x}{N})).}\end{equation}
\pat{From this we obtain that $\eqref{I} + \eqref{II} $ is equal to 
\begin{align*}
    \nonumber
&- \sum_{x = 1}^{N-2} \frac{1}{4a_x}D_{\nu^N_{\varrho(\cdot)}}^{x,x+1}(\sqrt f) -\sum_{x = 1}^{N-2} \frac {a_x}{4} D_{\nu^N_{\varrho(\cdot)}}^{x+1,x}(\sqrt f)\\
    &+ \sum_{x = 1}^{N-2} \frac{1}{4a_x}\int_{\Omega_N}  c_{x,x+1}(\eta) \left[f(\eta) - f(\eta^{x,x+1}) \right] d\nu^N_{\varrho(\cdot)}
    \\&+ \sum_{x = 1}^{N-2} \frac{a_x}{4} \int_{\Omega_N} c_{x+1,x}(\eta) \left[ f(\eta) - f(\eta^{x+1,x}) \right] d\nu^N_{\varrho(\cdot)}.
\end{align*}}
From the previous results  we obtain
\begin{align*}
     \langle \mcb{L}_{bulk,} \sqrt{f}, \sqrt{f} \rangle_{\nu^N_{\varrho(\cdot)}} &= - \frac{1}{4} \left[ \sum_{x=1}^{N-2} \left(1 + \frac{1}{a_x} \right) D_{\nu^N_{\varrho(\cdot)}}^{x,x+1}(\sqrt{f} ) + \sum_{x=1}^{N-2} \left(1 + a_x \right) D_{\nu^N_{\varrho(\cdot)}}^{x+1,x}(\sqrt{f}) \right] \\
    &+ \sum_{x=1}^{N-2}\frac{1}{4 }\int_{\Omega_N} c_{x+1,x}(\eta) \left[ f(\eta^{x+1,x}) - f(\eta)\right] (1-a_x)d\nu^N_{\varrho(\cdot)}\\
    &+ \sum_{x=1}^{N-2} \frac{1}{4} \int_{\Omega_N} c_{x,x+1}(\eta) \left[ f(\eta^{x,x+1}) - f(\eta)\right] \Big(1-\frac {1}{ a_x}\Big) d\nu^N_{\varrho(\cdot)}.
\end{align*}
By using the identity $(x-y)=(\sqrt x-\sqrt y)(\sqrt x+\sqrt y) $ for $x,y\geq 0$, Young's inequality   and the inequality $(x+y)^2 \leq 2(x^2 + y^2)$, we get that
\begin{align*}
     \langle \mcb{L}_{bulk} \sqrt{f}, \sqrt{f} \rangle_{\nu^N_{\varrho(\cdot)}} &\leq  - \frac{1}{4} \left[ \sum_{x=1}^{N-2} \left(\frac 12 + \frac{1}{a_x} \right) D_{\nu^N_{\varrho(\cdot)}}^{x,x+1}(\sqrt{f}, \nu^N_{\varrho(\cdot)} ) + \sum_{x=1}^{N-2} \left(\frac 12 + a_x \right) D_{\nu^N_{\varrho(\cdot)}}^{x+1,x}(\sqrt{f}) \right] \\
    &+ \sum_{x=1}^{N-2} \frac{1}{4} \int_{\Omega_N} c_{x+1,x}(\eta) \left[ f(\eta^{x+1,x}) + f(\eta)\right] (1-a_x)^2 d\nu^N_{\varrho(\cdot)}\\
    &+\sum_{x=1}^{N-2} \frac{1}{4}  \int_{\Omega_N} c_{x,x+1}(\eta) \left[ f(\eta) + f(\eta^{x,x+1})\right]  \Big(1-\frac {1}{ a_x}\Big)^2d\nu^N_{\varrho(\cdot)}.
\end{align*}
The proof now ends by  using the fact that $f$ is a density, the fact that $\varrho(\cdot)$ is Lipschitz continuous and bounded away from zero and one;  and the definition of $a_x$ given in \eqref{eq:a_x}.
\end{proof}
\subsection{Replacement Lemmas: at the {boundary} and at the bulk}
We start with the replacement lemma needed at the boundary for the Dirichlet case ($\theta<1$) and then we prove the replacement lemma  at the bulk needed in all cases ($\theta\in\mathbb R$).
\begin{lemma}\label{rep_lemma}
If $\theta < 1$, for any $t\in[0,T]$  and for $x=1$ it holds
\begin{equation}\label{eq:rep_bound}
\lim_{N\to+\infty}\mathbb E_{\mu_N}\Big[\Big| \int_0^t \Big(\varrho_- - \eta_{sN^2}(x) \Big)\,ds \Big| \Big]= 0.
\end{equation}
The same result is true for $x=N-1$ and \pat{with} $\varrho_+$ instead of $\varrho_-$.
 \end{lemma}
\begin{proof} 
Let $\nu^N_{\varrho(\cdot)}$ be the product measure defined in \eqref{reference_measure} and we assume that the profile satisfies the conditions exposed in the beginning of Section \ref{sec:dir_form}.
From entropy's and Jensen's inequalities, we  bound the expectation appearing in the statement of the lemma by 
\begin{align}\label{eq:imp}
     \frac{H(\mu_N | \nu^N_{\varrho(\cdot)})}{BN} + \frac{1}{BN} \log \Big( \mathbb{E}_{\nu^N_{\varrho(\cdot)}}\Big[e^{\Big|\int_0^t BN\Big(\eta_{sN^2}(1)-\varrho_- \Big) ds\Big|}  \Big] \Big),
\end{align}
for $B> 0$.
From the computations in the beginning of this section we know that  $H(\mu_N | \nu^N_{\varrho(\cdot)})\lesssim N$. Moreover, from the inequality  $e^{|x|} \leq e^x + e^{-x}$,  the fact that 
\begin{equation} \label{largedev}
\limsup_{N \rightarrow \infty} \dfrac{1}{N} \log (a_N + b_N ) = \max \left\lbrace \limsup_{N \rightarrow \infty} \dfrac{1}{N} \log (a_N  ), \; \limsup_{N \rightarrow \infty} \dfrac{1}{N} \log ( b_N ) \right\rbrace,
\end{equation}
 and  from Feynman-Kac's formula, \eqref{eq:imp} is bounded from above by  a constant times
\begin{align}\label{eq:imp_2}
\frac{1}{B}+ \frac{1}{BN} \int_0^t \sup_{f \textrm{density}} \Big\{ BN\langle \eta(1)-\varrho_-, f \rangle_{\nu^N_{\varrho(\cdot)}} + N^2 \langle \mcb{L}_N \sqrt f, \sqrt f \rangle_{\nu^N_{\varrho(\cdot)}} \Big\} ds.
\end{align}
Our goal now consists in estimating the term $\langle \eta(1)-\varrho_-, f \rangle_{\nu^N_{\varrho(\cdot)}}$.
We split the state space into disjoint hyper-planes with a fixed number of particles at the site $1$, i.e.  $\Omega_N = \cup_{i=0}^\alpha \Omega_i$, where $$\Omega_i = \{\eta \in \Omega_N \ | \ \eta(1) = i\}.$$ \pat{Fix $i\in\{0,\cdots, \alpha\} $.} Observe that from a change of variables $\xi=\eta^{0,1}$ and $\tilde \xi =\eta^{1,0}$, we get 
\begin{align} 
   & \int_{\Omega_i} \gamma \eta(1) f(\eta)d\nu^N_{\varrho(\cdot)}=\mathbb{1}_{\{1 \leq i \leq \alpha\}} \int_{\Omega_{i-1}} \gamma [\alpha - \eta(1)] f(\eta^{0,1}) \frac{\varrho(\frac{1}{N})}{1-\varrho(\frac{1}{N})} d\nu^N_{\varrho(\cdot)},  \label{A}
 \\
 &   \int_{\Omega_i} \epsilon [\alpha - \eta(1)] f(\eta)d\nu^N_{\varrho(\cdot)} =\mathbb{1}_{\{0 \leq i \leq \alpha-1\}} \int_{\Omega_{i+1}} \epsilon \eta(1) f(\eta^{1,0}) \frac{1-\varrho(\frac{1}{N})}{\varrho(\frac{1}{N})} d\nu^N_{\varrho(\cdot)}. \label{B}
\end{align}
 Using the identity $$\eta(1)-\varrho_- = \frac{-1}{\epsilon + \gamma} \left(-\gamma \eta(1) + \epsilon[\alpha - \eta(1)] \right),$$ now, instead of estimating $\langle \eta(1)-\varrho_-, f \rangle_{\nu^N_{\varrho(\cdot)}}$, we estimate 
 $(\epsilon+\gamma)\langle \eta(1)-\varrho_-, f \rangle_{\nu^N_{\varrho(\cdot)}}$, which is enough for our purposes.
 Note that we can write
\begin{align*}
    \langle -\gamma \eta(1) + \epsilon[\alpha - \eta(1)], f \rangle_{\nu^N_{\varrho(\cdot)}} &= \frac{1}{2} \sum_{i=0}^{\alpha} \int_{\Omega_i} \epsilon[\alpha - \eta(1)] f(\eta) d\nu^N_{\rho(\cdot)} - \frac{1}{2} \sum_{i=0}^\alpha \int_{\Omega_i} \gamma \eta(1)f(\eta) d\nu^N_{\varrho(\cdot)} \\
    &+ \frac{1}{2} \sum_{i=0}^{\alpha-1} \int_{\Omega_i} \epsilon[\alpha - \eta(1)] f(\eta) d\nu^N_{\rho(\cdot)} - \frac{1}{2} \sum_{i=1}^\alpha \int_{\Omega_i} \gamma \eta(1)f(\eta) d\nu^N_{\varrho(\cdot)}.
\end{align*}
From \eqref{A} and \eqref{B},  we get 
\pat{\begin{align*} \nonumber
    \langle -\gamma \eta(1) + \epsilon[\alpha - \eta(1)], f \rangle_{\nu^N_{\varrho(\cdot)}} &=\frac{1}{2} \sum_{i=0}^{\alpha} \int_{\Omega_i} \epsilon[\alpha - \eta(1)] [f(\eta) - f(\eta^{0,1})] d\nu^N_{\varrho(\cdot)} \\&- \frac{1}{2} \sum_{i=0}^\alpha \int_{\Omega_i} \gamma \eta(1)[f(\eta) - f(\eta^{1,0})] d\nu^N_{\varrho(\cdot)} \\ \nonumber
    &+ \frac{1}{2} \sum_{i=0}^{\alpha} \int_{\Omega_{i}} \left\{\frac{\epsilon(1-\varrho(\frac{1}{N}))}{\varrho(\frac{1}{N})} - \gamma \right\} \eta(1) f(\eta^{1,0})d\nu^N_{\varrho(\cdot)} \\&- \frac{1}{2} \sum_{i=0}^{\alpha} \int_{\Omega_i} \left\{  \frac{\gamma\varrho(\frac{1}{N})}{1-\varrho(\frac{1}{N})}  - \epsilon \right\}[\alpha - \eta(1)] f(\eta^{0,1}) d\nu^N_{\varrho(\cdot)}.
\end{align*}}
Since the  profile $\varrho(\cdot)$  is locally constant equal to $\frac{\epsilon}{\epsilon+\gamma}$ at $0$, then, the last line vanishes for $N$ sufficiently big. It remains to analyse the first line of last display. To that end, using the identity $(x-y)=(\sqrt x-\sqrt y)(\sqrt x+\sqrt y)$ for $x,y\geq0$, and Young's inequality, \pat{we get  for $A>0$:} 
\begin{align*}
    &\langle -\gamma \eta(1) + \epsilon[\alpha - \eta(1)], f \rangle_{\nu^N_{\varrho(\cdot)}} \\&\leq \frac{N^\theta}{4A} D^\ell_{\nu^N_{\varrho(\cdot)}}(\sqrt{f}) + \frac{A}{2} \int_{\Omega_N} \gamma \eta(1) [f(\eta) + f(\eta^{1,0})] + \epsilon[\alpha - \eta(1)] [f(\eta) + f(\eta^{0,1})] d\nu^N_{\varrho(\cdot)}.
\end{align*}
To conclude the proof it is enough to  invoke Lemma \ref{lem:Dir},  choose $A = \frac{N^\theta B}{4K N}$ and use the fact that $f$ is a density. Then, \eqref{eq:imp_2} vanishes by taking the limit in $N\to+\infty$ and then $B\to+\infty$. 
\end{proof}
\begin{lemma}\label{repl_lemma_ave}
For any $\theta\in\mathbb R$, for any  $t \in [0,T]$ and $z=1$ it holds
\begin{equation}
\lim_{\epsilon\to 0}\lim_{N\to+\infty}\mathbb{E}_{\mu_N} \left[\Big | \int_0^t \eta_{sN^2}(z) - \overrightarrow{\eta}^{\lfloor \epsilon N\rfloor}_{sN^2}(z) ds\Big| \right] =0,
\end{equation}
where, for $L\in\mathbb N$, $
 \overrightarrow{\eta}^L(z) :=\frac{1}{L}\sum_{y=z+1}^{z+L}\eta(y)$.
\end{lemma}
The proof is analogous for $z=N-1$ but taking instead the average to the left, i.e. $
 \overleftarrow{\eta}^L(z) :=\frac{1}{L}\sum_{y=z-L}^{z-1}\eta(y)$.
\begin{proof}
Fix $L\in\mathbb N$. 
Repeating the first steps of the previous proof, the expectation  in the statement of the lemma is bounded from above, for any  constant $B>0$, by  
\begin{align}\label{eq:imp_3}
\frac{1}{B} +t\sup_{f \textrm{density}} \Big\{ \langle \eta(z)- \overrightarrow{\eta}^L(z), f \rangle_{\nu^N_{\varrho(\cdot)}} + \frac{N}{B} \langle \mcb{L}_N \sqrt f, \sqrt f \rangle_{\nu^N_{\varrho(\cdot)}} \Big\},
\end{align}
Observe that
$$ \eta(z)- \overrightarrow{\eta}^L(z) =\frac{1}{L}\sum_{y=z+1}^{z+L} \left[\eta(z)-\eta(y) \right]=\frac{1}{L}\sum_{y=z+1}^{z+L}\sum_{x=z}^{y-1} \left[\eta(x)-\eta(x+1) \right].$$
Since  $\alpha [\eta(x)-\eta(x+1)]= c_{x,x+1}(\eta)-c_{x+1,x}( \eta)$, we will  analyze  $\langle c_{x,x+1}(\eta)-c_{x+1,x}(\eta), f \rangle_{\nu^N_{\varrho(\cdot)}} $. By writing the last integral as twice its half and  decomposing $\Omega_N = \cup_{i,j=0}^{\alpha} \Omega_{i,j}$, where $$\Omega_{i,j} := \{ \eta \in \Omega_N \ | \ \eta(x) = i , \,\eta(x+1) = j\},$$ we get that 
\pat{\begin{align} \nonumber
    \int (c_{x,x+1}(\eta)-c_{x+1,x}(\eta))f(\eta) d\nu^N_{\varrho(\cdot)} &=\frac{1}{2} \int (c_{x,x+1}(\eta)-c_{x+1,x}(\eta))f(\eta) d\nu^N_{\varrho(\cdot)} \\  \label{duble sum before exchanging variables}
   &+ \frac{1}{2} \sum_{i=1}^{\alpha} \sum_{j=0}^{\alpha-1} \int_{\Omega_{i,j}} c_{x,x+1}(\eta) f(\eta) d\nu^N_{\varrho(\cdot)} \\&- \frac{1}{2} \sum_{i=0}^{\alpha-1}
    \sum_{j=1}^{\alpha}
    \int_{\Omega_{i,j}} c_{x+1,x}(\eta)f(\eta) d\nu^N_{\varrho(\cdot)}.
\end{align}}
Recall \eqref{eq:change_variables} and  \eqref{eq:a_x}. Changing variables $\xi = \eta^{x,x+1}$ and \pat{$\tilde \xi = \eta^{x+1,x}$}  in \eqref{duble sum before exchanging variables} and then making a change of variables in the summations, we can rewrite the last display as
   \begin{align*} 
\frac{1}{2} \int (c_{x,x+1}(\eta)-c_{x+1,x}(\eta))f(\eta) d\nu^N_{\varrho(\cdot)}  +& \frac{1}{2} \sum_{i=0}^{\alpha-1} \sum_{j=1}^{\alpha} \int_{\Omega_{i,j}} c_{x+1,x}(\eta){a_x}f(\eta^{x+1,x}) d\nu^N_{\varrho(\cdot)} \\
    -& \frac{1}{2} \sum_{i=1}^{\alpha}
    \sum_{j=0}^{\alpha-1}
    \int_{\Omega_{i,j}} c_{x,x+1}(\eta)\frac {1}{a_x} f(\eta^{x,x+1}) d\nu^N_{\varrho(\cdot)}.
\end{align*}
By summing and subtracting appropriate terms we can rewrite the previous expression as  
\begin{align*}
    &\frac{1}{2} \int c_{x,x+1}(\eta)[f(\eta) - f(\eta^{x,x+1})] d\nu^N_{\varrho(\cdot)} - \frac{1}{2}\int c_{x+1,x}(\eta)[f(\eta) - f(\eta^{x+1,x})] d\nu^N_{\varrho(\cdot)}\\
    +& \frac{1}{2} \int c_{x+1,x}(\eta)(a_x-1) f(\eta^{x+1,x}) d\nu^N_{\varrho(\cdot)}- \frac{1}{2} \int c_{x,x+1}(\eta)\Big(\frac{1}{a_x}-1\Big)f(\eta^{x,x+1}) d\nu^N_{\varrho(\cdot)}
    \end{align*}
    From the identity $(x-y)=(\sqrt x-\sqrt y)(\sqrt x+\sqrt y)$ for $x,y\geq0$, and Young's inequality we bound the last display from above by 
    \begin{align*}    
    & \frac{1}{4A} D^{x,x+1}_{\nu^N_{\varrho(\cdot)}}(\sqrt f)+ \frac{1}{4A}D^{x+1,x}_{\nu^N_{\varrho(\cdot)}}(\sqrt f)\\
    +& \frac{A}{4} \int c_{x,x+1}(\eta)[\sqrt{f}(\eta) + \sqrt{f}(\eta^{x,x+1})]^2 d\nu_{\varrho(\cdot)}+ \frac{A}{4}\int c_{x+1,x}(\eta)[ \sqrt{f}(\eta) + \sqrt{f}(\eta^{x+1,x})]^2 d\nu^N_{\varrho(\cdot)}\\
    +& \frac{1}{2} \int c_{x+1,x}(\eta) (a_x-1)  f(\eta^{x+1,x}) d\nu^N_{\varrho(\cdot)}
    -  \frac{1}{2} \int c_{x,x+1}(\eta) \Big(\frac{1}{a_x}-1\Big) f(\eta^{x,x+1}) d\nu^N_{\varrho(\cdot)} .
         \end{align*}
From these computations, we see that 
\begin{align*}\label{eq:imp_4}
 \langle \eta(z)- \overrightarrow{\eta}^L(z) , f \rangle_{\nu^N_{\varrho(\cdot)}} \lesssim \frac{1}{L}\sum_{y=z+1}^{z+L}\sum_{x=z}^{y-1}\Big\{\frac{1}{A} \Big(D^{x,x+1}_{\nu^N_{\varrho(\cdot)}}(\sqrt f)+D^{x+1,x}_{\nu^N_{\varrho(\cdot)}}(\sqrt f)\Big)+A +\Big|\varrho \left( \tfrac{x}{N} \right) - \varrho \left( \tfrac{x+1}{N}\right)\Big|\Big\}.
\end{align*}
Now we invoke Lemma \ref{lem:Dir} and we make the choice above  $A=B/4NK$. Taking  $L=\lfloor \epsilon N\rfloor$, since the  profile is Lipschitz continuous we have that $$\frac{1}{L}\sum_{y=z+1}^{z+L} \sum_{x=z}^{z+L-1}|\varrho(\tfrac xN)-\varrho(\tfrac {x+1}{N})|\lesssim \epsilon.$$ To conclude the proof  it is enough to take  the limit as $N\to+\infty$, $\epsilon\to 0$ and then $B\to+\infty$. 
\end{proof}
\acks{
The authors  thank  FCT/Portugal for financial support through the project 
UID/MAT/04 459/2013. B.S.  also thanks CAMGSD  for financial support.  This material is based upon work supported by the National Science Foundation under Grant No. DMS-1928930 while C.F. participated in a program hosted by the Mathematical Sciences Research Institute in Berkeley, California, during the Fall 2021 semester. This project has received funding from the European Research Council (ERC) under  the European Union's Horizon 2020 research and innovative programme (grant agreement n. 715734).
}


\end{document}